\documentclass[10pt]{amsart}
\usepackage[usenames,dvipsnames]{pstricks}
\usepackage[mathscr]{eucal}
\usepackage{amsfonts,amsmath,amssymb,amsthm,amscd,amsxtra}
\usepackage{amssymb,amstext,amsmath,amscd,amsthm,amsfonts,enumerate,graphicx,latexsym}
\usepackage{enumerate,verbatim}
\usepackage[all,2cell,ps]{xy}
\usepackage[notcite, notref]{}
\usepackage[pagebackref]{hyperref}
\usepackage{todonotes}

\usepackage[all,2cell,ps]{xy}
\usepackage[pagebackref]{hyperref}

\usepackage[usenames,dvipsnames]{pstricks}
\usepackage[mathscr]{eucal}
\usepackage{amsfonts,amsmath,amssymb,amsthm,amscd,amsxtra}
\usepackage{enumerate,verbatim}
\usepackage[all,2cell,ps]{xy}
\usepackage[notcite, notref]{}
\usepackage[pagebackref]{hyperref}
\usepackage{todonotes}

\usepackage[all,2cell,ps]{xy}
\usepackage[pagebackref]{hyperref}

\newtheorem{thm}{Theorem}[section]
\newtheorem{cor}[thm]{Corollary}
\newtheorem{lem}[thm]{Lemma}
\newtheorem{prop}[thm]{Proposition}
\newtheorem{defn}[thm]{Definition}

\newtheorem{exam}[thm]{Example}

\newtheorem*{acknowledgement}{Acknowledgements}
\newtheorem*{Theorem}{Theorem [Peskine-Szpiro]}

\def\Ext{\operatorname{\mathsf{Ext}}}
\def\Tor{\operatorname{\mathsf{Tor}}}
\def\CI{\operatorname{\mathsf{CI-dim}}}

\def\Hom{\operatorname{\mathsf{Hom}}}
\def\pd{\operatorname{\mathsf{pd}}}
\def\id{\operatorname{\mathsf{id}}}
\def\hd{\operatorname{\mathsf{H-dim}}}
\def\dim{\operatorname{\mathsf{dim}}}
\def\grade{\operatorname{\mathsf{grade}}}
\def\depth{\operatorname{\mathsf{depth}}}
\def\Tr{\mathsf{Tr}\hspace{0.01in}}

\DeclareMathOperator{\Ann}{Ann}
\DeclareMathOperator{\Ass}{Ass}

\newcommand{\rank}{\mbox{rank}\, }
\newcommand{\Spec}{\mbox{Spec}\, }

\newcommand{\Supp}{\mbox{Supp}\, }

\def\depth{\operatorname{\mathsf{depth}}}
\def\dim{\operatorname{\mathsf{dim}}}
\renewcommand{\Im}{\mbox{Im}\, }

\newcommand{\gd}{\mbox{G-dim}}
\newcommand{\gkd}{\mbox{G}_{K}\mbox{-dim}\,}
\newcommand{\gc}{\mbox{G}_{C}\mbox{-dim}\,}

\newcommand{\gk}{\mbox{G}_{K}\mbox{-}\,}
\newcommand{\ZZ}{\mathbb{Z}}

\def\gd{\operatorname{\mathsf{G-dim}}}

\def\ci{\operatorname{\mathsf{CI-dim}}}

\newcommand{\E}{\mbox{E}}

\renewcommand{\H}{\mbox{H}}

\newcommand{\g}{\mbox{G}}

\newcommand{\fa}{\mathfrak{a}}

\newcommand{\fm}{\mathfrak{m}}
\newcommand{\fp}{\mathfrak{p}}
\newcommand{\fq}{\mathfrak{q}}

\title[associated primes and syzygies of linked modules]
{associated primes and syzygies of linked modules}
\author[Celikbas, Dibaei, Gheibi, Sadeghi, Takahashi]
{Olgur Celikbas, Mohammad T. Dibaei, \\ Mohsen Gheibi, Arash Sadeghi and Ryo Takahashi}

\address{Olgur Celikbas \\
Department of Mathematics \\
West Virginia University\\
Morgantown, WV 26506-6310, U.S.A}
\email{olgur.celikbas@math.wvu.edu}

\address{Mohammad T. Dibaei \\
Faculty of Mathematical Sciences and Computer\\
Kharazmi University\\
and School of Mathematics Institute for Research in Fundamental Sciences (IPM)\\
P.O. Box19395-5746, Tehran, Iran}
\email{dibaeimt@ipm.ir}

\address{Mohsen Gheibi \\
Department of Mathematics\\
University of Nebraska\\
Lincoln, NE, 68588,USA}
\email{mohsen.gheibi@huskers.unl.edu}

\address{Arash Sadeghi \\
 School of Mathematics\\
Institute for Research in Fundamental Sciences\\
(IPM), P.O. Box: 19395-5746, Tehran, Iran}
\email{sadeghiarash61@gmail.com}

\address{Ryo Takahashi\\
Graduate School of Mathematics, Nagoya University, Furocho, Chikusaku,\\
Nagoya 464-8602, Japan}
\email{takahashi@math.nagoya-u.ac.jp}

\date{December 14, 2016}

\keywords{Linkage of Modules, Geometrically linked ideals, Huneke-Wiegand conjecture, torsion-free module, vanishing of $\Tor$, homological dimension.\\
M. T. Dibaei was supported in part by a grant from IPM (No. 94130110). M. Gheibi was supported in part by a grant from IPM (No. $92130025$).
A. Sadeghi's research was supported by a grant from IPM. R. Takahashi was partly supported by JSPS Grant-in-Aid for Scientific Research (C) 25400038.}


\begin{document}

\begin{abstract}
Motivated by the notion of geometrically linked ideals, we show that over a Gorenstein local ring $R$, if a Cohen-Macaulay $R$-module $M$ of grade $g$ is linked to an $R$-module $N$ by a Gorenstein ideal $c$, such that $\Ass_R(M)\cap \Ass_R(N)=\emptyset$, then $M\otimes_RN$ is isomorphic to direct sum of copies of $R/\fa$, where $\fa$ is a Gorenstein ideal of $R$ of grade $g+1$. We give a criterion for the depth of a local ring $(R,\fm,k)$ in terms of the homological dimensions of the modules linked to the syzygies of the residue field $k$. As a result we characterize a local ring $(R,\fm,k)$ in terms of the homological dimensions of the modules linked to the syzygies of $k$.
\end{abstract}

\maketitle

\section{Introduction}
 Peskine and Szpiro \cite{PS}, introduced the notion of linkage and geometric linkage of projective algebraic varieties; recall that two unmixed projective varieties $X$ and $Y$ of codimension $g$ are \emph{geometrically linked} if $X$ and $Y$ have no common components and the scheme theoretic union $X\cup Y$ is a complete intersection. More generally, two non-zero proper ideals $I$ and $J$ in a local ring $R$ are said to be \emph{linked} by an ideal $c$ contained in $I\cap J$ if $I=(c :_RJ)$ and $J=(c:_RI)$ (denoted $I \underset{c}\sim J$). If in addition, $\Ass_R(R/I)\cap \Ass_R(R/J)=\emptyset$, then $I$ and $J$ are said to be \emph{geometrically linked} by $c$. If $c=0$ we just say $I$ is linked (or geometrically linked) to $J$. In their landmark paper \cite{PS}, Peskine and Szpiro prove the following:
 \begin{Theorem}\label{PS1} Let $R$ be a Gorenstein local ring and let $I$ be a Cohen-Macaulay ideal of grade $g$ (i.e. $R/I$ is Cohen-Macaulay ring). Suppose that $I$ is geometrically linked to an ideal $J$ by a Gorenstein ideal $c$ (i.e. $R/c$ is a Gorenstein ring). Then $I+J$ is a Gorenstein ideal of grade $g+1$.
 \end{Theorem}
 Therefore a new Gorenstein ideal can be obtained from the sum of geometrically linked ideals.

 Another important result about geometrically linked ideals was proved by Ulrich in \cite{U}. Recall that an ideal $I$ of a local ring $R$ is called \emph{licci} (in the linkage class of a complete intersection) if there exists a chain of linked ideals $I\underset{c_1}\sim I_1 \underset{c_2}\sim \cdots \underset{c_n}\sim I_n$, where $c_1,\cdots,c_n$, and $I_n$ can be generated by a regular sequence.  Ulrich showed that if $R$ is a Gorenstein local ring, then the sum of two geometrically linked licci ideals of $R$ is again a licci ideal. Also Huneke \cite{H1}, showed that if $R$ is Gorenstein and $I$ is a licci ideal then the Koszul homologies $\H_i(I)$ of $I$ are Cohen-Macaulay, for all $i\geq 0$. So over a Gorenstein local ring, finding geometrically linked ideals is useful and important for finding other nice ideals. For more results about geometrically linked ideals, see \cite{PS}, \cite{U}, \cite{J}, \cite{H} and \cite{Sc}.

Linkage of ideals has been generalized to modules by Yoshino and Isogawa \cite{YI}, by Marsinkovsky and Strooker \cite{MS} and by Nagel \cite{N}, in different ways and, based on these generalizations, several works have been done on studying classic results of linkage in the context of modules; see for example \cite{DGHS}, \cite{DS}, \cite{DS2} and \cite{IT}. In this paper, we are interested in linkage of modules in the sense of \cite{MS}.

Let $R$ be a commutative Noetherian local ring, $M$ a finitely generated $R$-module, and $F_1 \overset{f}\rightarrow F_0 \rightarrow M \rightarrow 0 $ a minimal free presentation of $M$. The image of $f$ is called the first syzygy of $M$ and it is denoted $\Omega M$. The transpose of $M$ (denoted $\Tr M$) is defined to be the cokernel of the induced map $f^*:\Hom_R(F_0,R)\rightarrow \Hom_R(F_1,R)$.
According to \cite{MS}, $M$ is \emph{horizontally linked} to an $R$-module $N$ if $M\cong \Omega\Tr N$ and $N\cong \Omega\Tr M$.
The transpose and syzygy functors and their compositions already have been studied by Auslander and Bridger; see \cite{AB}.
Two $R$--modules $M$ and $N$ are said to be \emph{linked by an ideal} $c\subseteq \Ann_R(M)\cap\Ann_R(N)$ (denoted $M\underset{c}\sim N$), if $M$ is horizontally linked to $N$ over $R/c$. It turns out that two ideals $I$ and $J$ of $R$ are linked by an ideal $c\subseteq I\cap J$ if and only if the $R$-module $R/I$ is linked to $R/J$ by $c$; see \cite{MS}.

According to the notion of the geometrically linked ideals, the first purpose of this paper is to study linked modules that have no common associated primes. In section 3, we give a generalization of the above theorem. More precisely, we prove the following:

\begin{thm}\label{PS2} Let $R$ be a Cohen-Macaulay local ring, $K$ a semidualizing $R$-module and $M$ a $\gk$perfect $R$-module of grade $n<\dim R$. Assume $M$ is linked to $N$ by a $\gk$Gorenstein ideal $c$, and $\Ass_R(M)\cap \Ass_R(N)=\emptyset$. Then $M\otimes_RN$ is free over $S:=R/(\Ann_R(M)+\Ann_R(N))$ and $\Ann_R(M)+\Ann_R(N)$ is a $\gk$Gorenstein ideal of grade $n+1$.
\end{thm}

For definition of semidualizing module, $\gk$perfect module, and $\gk$Gorenstein ideal, see section 2.
Along with studying the results of geometrically linked ideals in the context of modules, also we prove the following:

\begin{thm}\label{Int1} Let $R$ be a generically Gorenstein unmixed local ring of dimension $d\geq 1$, and let $M$ be a finitely generated $R$-module. Assume $M$ is horizontally linked to $\lambda M$. Then the following statements are equivalent:
 \begin{enumerate}[\rm(i)]
  \item $\Ass_R(M)\cap \Ass_R(\lambda M)=\emptyset$.
  \item $\Ann_R(M)$ is geometrically linked to $\Ann_R(\lambda M)$.
  \item $\Ann_R(M)$ is linked to $\Ann_R(\lambda M)$ and $\Tor^R_1(M,\lambda M)=0$.
  \item $M$ is free over $R/\Ann_R(M)$ and $\Ext^1_R(M,M)=0$.
 \end{enumerate}
\end{thm}

Our next aim in this paper is to study the properties of a local ring $(R,\fm,k)$ by looking at linkage of syzygy modules of $k$.
Dutta \cite{D} showed that if some syzygy module of the residue field $k$ has a non-zero direct summand of finite projective dimension, then $R$ is regular. Motivated by Dutta's Theorem, a natural question is what properties does the ring $R$ have if an $R$-module linked to a syzygy of $k$, has a non-zero direct summand of finite projective dimension? We prove the following result in section 5:

\begin{thm} \label{Int4} Let $(R,\fm,k)$ be a local ring and let $n\geq 0$ be an integer. Let $X$ be a non-zero direct summand of $\lambda\Omega^nk$. Then the following statements are equivalent.
\begin{enumerate}[\rm(i)]
\item $\pd(X) < \infty$.
\item $\pd(X)= n$.
\item $\depth R > n$.
\end{enumerate}
\end{thm}

Also, we show that a similar assertion holds if a module horizontally linked to $\Omega^nk$ has finite complete intersection dimension or finite Gorenstein dimensions.

\section{Preliminaries}

In the present section we recall some preliminary results from the literature that we need for the rest of the paper.

Throughout $R$ is a commutative Noetherian local ring and all modules over $R$ are assumed to be finitely generated.
For a finitely generated $R$--module $M$, let $F_1\rightarrow F_0\rightarrow M\rightarrow 0$
be a minimal free presentation of $M$. Then by applying $(-)^*=\Hom_R(-,R)$, there exists an exact sequence
\begin{equation}\tag{2.1}\label{eq1}
0\longrightarrow M^*\rightarrow F_0^*\overset{f^*}\rightarrow F_1^*\rightarrow \Tr M\rightarrow 0.
\end{equation}
If $M$ is stable (i.e. has no free direct summand) then $\Tr\Tr M\cong M$, and \ref{eq1} gives a minimal presentation of $\Tr M$
(see \cite[Theorem 32.13]{AF}). As in \cite{MS}, for a non-zero $R$-module $M$ we will set $\lambda M $ to be the first syzygy module of $\Tr M$ i.e. $\lambda M= \Omega\Tr M$. Hence $M$ is horizontally linked to $\lambda M$ if and only if $M \cong\lambda^2 M$. In this case, we briefly say $M$ is horizontally linked. Also, if an ideal $I$ is linked (geometrically linked) to $(0:_RI)$, we just say $I$ is linked (geometrically linked) ideal of $R$. Note that $M$ is a free $R$-module if and only if $\lambda M=0$.

In the following, we recall some basic properties of linkage of modules that we need for the rest of paper.

\begin{prop}\label{MS3}\cite[Proposition 3]{MS} Let $R$ be a local ring and let $M$ be a horizontally linked $R$-module.
Then $M$ and $\lambda M$ are stable $R$-modules.
\end{prop}

Let $R$ be a ring. An $R$-module $M$ is called a syzygy module if there exists an injective $R$-homomorphism $\iota : M \rightarrow F$, where $F$ is a free $R$-module.

\begin{prop}\label{ExtTr} Let $R$ be a ring and let $M$ be an $R$-module. Then $M$ is a syzygy module if and only if $\Ext^1_R(\Tr M, R)=0$ (see \cite[Theorem 2.17]{AB}).
\end{prop}

\begin{thm}\label{MS2}\cite[Theorem 2]{MS} Let $R$ be a local ring and let $M$ be an $R$-module. Then the following statements are equivalent.
\begin{enumerate}[\rm(i)]
\item $M$ is horizontally linked.
\item $M$ is stable and \emph{$\Ext^1_R(\Tr M,R) = 0$}.
\item $M$ is stable and a syzygy module.
\end{enumerate}
\end{thm}

\begin{prop}\label{MS4}\cite[Theorem 3.5]{EG} Let $R$ be a generically Gorenstein local ring, i.e., $R_{\fp}$ is Gorenstein ring, for all $\fp\in \Ass R$.
An $R$-module $M$ is a syzygy module if and only if $\Ass_R(M) �\subseteq \Ass(R)$.
\end{prop}

Let $R$ be a local ring. Two $R$-modules $M$ and $N$ are called \emph{stably isomorphic} (denoted $M\underset{st}\cong N$) if there exist
free $R$-modules $F$ and $G$ such that $M\oplus F \cong N\oplus G$. The next proposition follows from the exact sequence (\ref{eq1}) and Theorem \ref{MS2}.

\begin{prop}\label{fund} Let $R$ be a local ring and $M$ be an $R$-module. Then
\begin{enumerate}[\rm(i)]
\item $\Omega\lambda M \underset{st}\cong M^{\ast}$, and if $M$ is horizontally linked, then $\Omega M \cong (\lambda M)^*$,
\item $(\lambda\Omega^nM)^*\underset{st}\cong \Omega^{n+1}M$ and $\lambda^2\Omega^n M \underset{st}\cong \Omega^n M$, for all $n\geq 1$.
\end{enumerate}
\end{prop}

The following result describes an interesting relation between the set of associated primes of horizontally linked modules and associated primes of the base ring.

\begin{prop}\label{F1}\cite[Proposition 6]{MS} \emph{Let $R$ be an unmixed (i.e. $\Ass(R)$ consists of minimal prime ideals) local ring and let $M$ be a horizontally linked $R$-module.
Then $\Ass_R(M)\cup \Ass_R(\lambda M) = \Ass(R)$.}
\end{prop}

Next, we recall the definition of Gorenstein dimension which has been introduced by Auslander and Bridger \cite{AB}.

\begin{defn}
\emph{A finitely generated $R$-module $M$ is called \emph{totally reflexive} or of $\g$-\emph{dimension zero} if the natural homomorphism $M\to M^{\ast\ast}$ is an isomorphism and $\Ext_R^i(M,R)=0=\Ext_R^i(M^\ast,R)$ for all $i>0$.}
\end{defn}

We recall that the infimum of $n$ for which there exists an exact sequence $$0 \to X_{n} \to \dots  \to X_{0} \to M\to 0,$$
such that each $X_{i}$ is totally reflexive, is called the \emph{Gorenstein dimension} of $M$. If $M$ has Gorenstein dimension $n$, we write $\gd(M)=n$.
Therefore $M$ is totally reflexive if and only if $\gd(M)\leq 0$, where it follows by convention that $\gd(0)=-\infty$.


More generally, let $M$ and $K$ be $R$--modules, and set $M^\dagger=\Hom_R(M,K)$. The module $M$ is called \emph{$K$-reflexive} if the canonical map $M\rightarrow M^{\dagger\dagger}$ is bijective. The Gorenstein dimension has been extended to $\gk$dimension by Foxby in \cite{F} and by Golod in \cite{G}.

\begin{defn}
\emph{The module $M$ is said to have \emph{$\gk$dimension zero} if}
\begin{enumerate}[\rm(i)]
            \item{\emph{$M$ is $K$-reflexive,}}
             \item{\emph{$\Ext^i_R(M,K)=0$}, for all $i>0$,}
              \item{\emph{$\Ext^i_R(M^{\dagger},K)=0$}, for all $i>0$.}
\end{enumerate}
\end{defn}
A $\gk$resolution of a finite $R$--module $M$ is a right acyclic complex of modules of $\gk$dimension zero whose $0$th
homology module is $M$.  The module $M$ is said to have finite $\gk$dimension, denoted by $\gkd_R(M)$, if
it has a $\gk$resolution of finite length.

\begin{defn}
\emph{An $R$--module $K$ is called \emph{semidualizing} (or \emph{suitable}), if}
\begin{enumerate}[\rm(i)]
    \item\emph{the homothety morphism \emph{$R\rightarrow\Hom_R(K,K)$} is an isomorphism,}
    \item\emph{$\Ext^i_R(K,K)=0$} for all $i>0$.
\end{enumerate}
\end{defn}
Semidualizing modules are studied in \cite{F} and \cite{G}.
It is clear that $R$ itself is a semidualizing $R$--module. Also if $R$ is Cohen-Macaulay then its canonical module (if it exists)
is a semidualizing module. We recall the following definitions from \cite{G}.
\begin{defn}
\emph{An $R$--module $M$ is called \emph{$\gk$perfect} if $\grade_R(M)=\gkd_R(M)$. An ideal $I$ is called $\gk$perfect if $R/I$ is
$\gk$perfect $R$--module. An $R$--module $M$ is called \emph{$\gk$Gorenstein} if $M$ is $\gk$perfect
and $\Ext^n_R(M,K)$ is cyclic, where $n=\gkd_R(M)$. An ideal $I$ is called $\gk$Gorenstein if $R/I$ is $\gk$Gorenstein $R$--module.}
\end{defn}
Note that if $K$ is a semidualizing $R$--module and $I$ is a $\gk$Gorenstein ideal of $\gk$dimension $n$, then
$\Ext^n_R(R/I,K)\cong R/I$(see \cite[10]{G}).

\begin{lem}\label{G1} \cite[Corollary]{G}
Let $I$ be an ideal of $R$. Assume that $K$ is an $R$--module and that $n$ is a fixed integer. If \emph{$\Ext^j_R(R/I,K)=0$} for
all $j\neq n$ then there is an isomorphism of functors \emph{$\Ext^i_{R/I}(-,\Ext^n_R(R/I,K))
\cong\Ext^{n+i}_R(-,K)$} on the category of $R/I$--modules for all $i\geq0$.
\end{lem}

\begin{thm}\label{G2}
\cite[Proposition 5]{G}. Let $I$ be a \emph{$\gk$}perfect ideal, and let K be a semidualizing
$R$--module. Set $C=\Ext^{\tiny{\grade(I)}}_R(R/I,K)$. Then the following statements hold.
\begin{enumerate}[\rm(i)]
        \item {$C$ is a semidualizing $R/I$--module.}
         \item {If $M$ is a $R/I$--module then \emph{$\gkd_R(M)<\infty$} if and only if \emph{$\gc_{R/I}(M)<\infty$}, and
         if these dimensions are finite then
\emph{$\gkd_R(M)=\grade(I)+\gc_{R/I}(M)$.}}
\end{enumerate}
\end{thm}

\section{Linkage of modules and geometrically linked ideals}

The purpose of this section is to study the conditions under which the annihilator of a linked module is a geometrically linked ideal. We start with the following well-known proposition which gives a characterization of geometrically linked ideals, and motivation for the results of sections 3 and 4.

\begin{prop}\label{P1} Let $R$ be an unmixed ring, $I$ an ideal of grade zero and linked to an ideal J. Then the following
statements are equivalent.
\begin{enumerate}[\rm(i)]
\item $I$ is geometrically linked to $J$.
\item $\grade_R(I+J)>0$.
\item $I\cap J=0$.
\item $\Tor^R_1(R/I,R/J)=0$.
\item $\Ext^1_R(R/I,R/I)=0$.
\item $\Ext^1_R(R/J,R/J)=0$.
\end{enumerate}
\end{prop}
\begin{proof} The equivalence of (i), (ii), (iii) and (iv) can be easily checked; see \cite{PS} and \cite[Proposition 2.3]{Sc}.

(v)$\Longleftrightarrow$(i) Consider the exact sequence $0\longrightarrow I\longrightarrow R\longrightarrow R/I \longrightarrow 0$. Applying $\Hom_R(-,R/I)$,
we get the exact sequence $$0\rightarrow \Hom_R(R/I,R/I)\overset{\alpha}\rightarrow \Hom_R(R,R/I)\rightarrow \Hom_R(I,R/I)\rightarrow
\Ext^1_R(R/I,R/I)\rightarrow 0.$$ As $\alpha$ is an isomorphism, one has $\Ext^1_R(R/I,R/I)\cong \Hom_R(I,R/I)\cong \Hom_R(\Hom_R(R/J,R),R/I)$.
It follows that $\Ext^1_R(R/I,R/I)=0$ if and only if $\Hom_R(\Hom_R(R/J,R),R/I)=0$ if and only if $\Ass_R(R/I) \cap \Ass_R(R/J) = \emptyset$;
see \cite[Exercise 1.2.27]{BH}.
\end{proof}

We will need the following elementary lemma:

\begin{lem}\label{L1} Let $R$ be an unmixed local ring, and let $M$ be a horizontally linked $R$--module.
Then \emph{$\Ass_R(M) =\Ass_R (R/\Ann_R(M)) $}.
\end{lem}
\begin{proof} Set \emph{$I=\Ann_R(M)$}. There is an inclusion $R/I\hookrightarrow \Hom_R(M,M)$ which implies that $\Ass_R(R/I) \subseteq \Ass_R(M)$.
Conversely let $\fp \in \Ass_R(M)$. By \ref{F1}, $\Ass_R(M) \subseteq \Ass(R)$, and so $\fp\in\Ass(R)$. Hence $\fp$ is minimal over $I$ and so $\fp \in \Ass_R(R/I)$.
\end{proof}

\begin{prop}\label{T1} Let $R$ be an unmixed local ring and let $M$ be an $R$-module. Assume $M$ is horizontally linked to $\lambda M$, and that $\Ass_R(M)\cap \Ass_R(\lambda M)=\emptyset$. Then $M$ and $\lambda M$ are free over $R/\Ann_R(M)$ and $R/\Ann_R(\lambda M)$, respectively, and $\Ann_R(M)$ is geometrically linked to $\Ann_R(\lambda M)$.
\end{prop}
\begin{proof} Let  $F_1 \rightarrow F_0 \rightarrow M \rightarrow 0$ be a minimal free presentation of $M$.
Then we have an exact sequence $0 \rightarrow M^* \rightarrow F_0^* \rightarrow \lambda M \rightarrow 0$. Set $I=\Ann_R(M)$ and $J=\Ann_R(\lambda M)$.
Applying $-\otimes_RR/J$ gives the exact sequence $$ M^* \otimes_R R/J \overset{f}\longrightarrow F_0^* \otimes_R R/J \longrightarrow \lambda M\otimes_R R/J\longrightarrow 0.$$
Since $\Ass_R(M)\cap \Ass_R(\lambda M)=\emptyset$, we have $\Ass_R(M) \cap \Ass_R(R/J) = \emptyset$ by \ref{L1}. Hence we have $\Supp_R (M^* \otimes_R R/J) \cap \Ass_R (R/J) = \emptyset$. This implies that $\Hom_R(M^* \otimes_R R/J,R/J)=0$; see \cite[Exercise 1.2.27]{BH}. Hence $f=0$ and the exact sequence above implies that $F_0^*\otimes_R R/J\cong \lambda M \otimes_R R/J \cong \lambda M$. Symmetrically, $F_0\otimes_RR/I\cong M$.

For the second part, note that as $M$ is a free $R/I$-module we may assume that $M=R/I$. Since $M$ is horizontally linked, $I$ is linked to $J$; see \cite[Lemma 3]{MS}, and as $\Ass_R(R/I) \cap \Ass_R(R/J) = \emptyset$, $I$ and $J$ are geometrically linked.
\end{proof}

Now we prove Theorem \ref{PS2}

\begin{proof}[Proof of Theorem \ref{PS2}] Set $\overline{R}:=R/c$. As $c$ is a $\gk$Gorenstein ideal, $\overline{R}$ is a Cohen-Macaulay ring and $\Ext^n_R(\overline{R},K)\cong \overline{R}$. Note that $\grade(M)=\grade(c)$ by \cite[Lemma 5.8]{DS}. Therefore $M$ and hence $N$, are totally reflexive $\overline{R}$-modules, by Theorem \ref{G2}. Note that $\fp \in \Ass_R(M)$ if and only if $\fp/c \in \Ass_{\overline{R}}M$. Hence by Proposition \ref{T1}, $M$ and $N$ are free over $\overline{R}/\Ann_{\overline{R}}(M)$ and $\overline{R}/\Ann_{\overline{R}}(N)$, respectively, and $\Ann_{\overline{R}}(M)$ is geometrically linked to $\Ann_{\overline{R}}(N)$. Thus $M\otimes_{\overline{R}}N$ is a free $T$-module, where $T=\frac{\overline{R}}{(\Ann_{\overline{R}}(M)+\Ann_{\overline{R}}(N))}$. But as $M$ is totally reflexive, it follows that $\overline{R}/\Ann_{\overline{R}}(M)$ and $\overline{R}/\Ann_{\overline{R}}(N)$ are totally reflexive $\overline{R}$-modules. Set $I=\Ann_{\overline{R}}(M)$ and $J=\Ann_{\overline{R}}(N)$. By \ref{P1} (iii), $I\cap J=0$ and hence one has $\grade_{\overline{R}}(I+J)>0$, by \ref{P1} (ii). There is an exact sequence $0\rightarrow \overline{R}\rightarrow \overline{R}/I\oplus \overline{R}/J\rightarrow T\rightarrow 0$. Since $\overline{R}/I$ and $\overline{R}/J$ are totally reflexive $\overline{R}$-modules, it follows that $\Ext^i_{\overline{R}}(T,\overline{R})=0$ for all $i\neq 1$. Also, by applying $(-)^*=\Hom_{\overline{R}}(-,\overline{R})$ to the above exact sequence, we get the commutative diagram with exact rows:
$$\begin{CD}
&&&&&&&&\\
\ \ &&&& 0@>>>(\overline{R}/I)^*\oplus (\overline{R}/J)^* @>>> (\overline{R})^* @>>> \Ext^1_{\overline{R}}(T,\overline{R}) @>>>0&  \\
\ \ &&&&&& @V{\cong}VV  @V{\cong}VV \\
\ \ &&&& 0@>>> I+J @>>> \overline{R}  @>>> T @>>>0.&\\
\end{CD}$$
It follows that $\Ext^1_{\overline{R}}(T,\overline{R})\cong T$. Therefore $\Ann_R(M)+\Ann_R(N)$ is a $\gk$Gorenstein ideal of grade $n+1$ by Lemma \ref{G1}.
\end{proof}

Let $R$ be an unmixed local ring of dimension $d\geq 1$ and let $M$ be a horizontally linked $R$-module. It follows from \ref{T1} and \ref{P1} that if $\Ass_R(M)\cap\Ass_R(\lambda M)=\emptyset$, then $\Tor^R_1(M,\lambda N)=0$. Consider the exact sequence
$$0\longrightarrow \Tor^R_1(M,\lambda M)\longrightarrow M\otimes_RM^* \longrightarrow M\otimes_RF_0^*.$$
Therefore if $\Tor^R_1(M,\lambda M)=0$ then $\depth_R(M\otimes_R M^*)>0$. In particular, if $R$ is a domain, then $\Tor^R_1(M,\lambda M)=0$ if and only if $M \otimes_R M^*$ is a torsion-free $R$-module if and only if $\lambda M \otimes_R (\lambda M)^*$ is a torsion-free $R$-module. Huneke and Wiegand \cite{HW} conjectured that if $R$ is a one-dimensional Gorenstein domain and $M$ is a finitely generated torsion-free $R$-module such that
$M\otimes_RM^*$ also is torsion-free then $M$ is free. The Huneke-Wiegand conjecture holds over one-dimensional hypersurface rings and also some
classes of numerical semigroup rings; see \cite{HW}, \cite{GSL}, and \cite{GTTT}, and it is still open for complete intersection rings of higher codimension. Therefore, the Huneke-Wiegand conjecture states that over a one-dimensional Gorenstein local domain $R$, there exists no horizontally linked $R$-module $M$ such that $\Tor^R_1(M,\lambda M)=0$.

Without assuming $R$ is domain, we are interested in conditions on the base ring $R$ and a horizontally linked $R$-module $M$ such that vanishing of
$\Tor^R_1(M,\lambda M)$ implies that $\Ass_R(M)\cap\Ass_R(\lambda M)=\emptyset$, and hence $M$ is a free $R/\Ann_R(M)$-module. However, in general vanishing of $\Tor^R_1(M,\lambda M)$ does not guarantee that $\Ass_R(M)\cap\Ass_R(\lambda M)=\emptyset$; see Example \ref{Ex2}.

An $R$-module $M$ is called \emph{generically free} if $M_{\fp}$ is a free $R_{\fp}$-module, for all $\fp\in \Ass(R)$.

\begin{lem}\label{P1-1} Let $R$ be a generically Gorenstein unmixed local ring and let $M$ be a horizontally linked $R$--module.
Then the following statements are equivalent.
\begin{enumerate}[\rm(i)]
\item \emph{$\Supp_R(\Tor^R_1(M,\lambda M))\cap\Ass(R)=\emptyset$}.
\item $M$ is generically free.
\end{enumerate}
\end{lem}
\begin{proof}  (i)$\Rightarrow$(ii). Let $\fp\in \Ass(R)$. Since $R$ is unmixed, $\fp$ is a minimal prime ideal of $R$, and so $R_{\fp}$ is Artinian Gorenstein local ring.  By assumption $\Tor^{R_\fp}_1(M_\fp,(\lambda M)_\fp)=0$ and since $(\lambda M)_\fp$ is stably isomorphic to $\lambda_{R_{\fp}} M_\fp$, we get $\Tor^{R_\fp}_1(M_\fp,\lambda_{R_{\fp}} M_\fp)=0$. Hence it is enough to show that if $R$ is an Artinian Gorenstein local ring and $M$ is an $R$-module such that $\Tor^R_1(M,\lambda M)=0$ then $M$ is free. Indeed, since $\Tor^R_1(M,\lambda M)=0$ so $\Hom_R(\Tor^R_1(M,\lambda M),R)=0$ and as $R$ is Artinian Gorenstein ring, the isomorphism $\Hom_R(\Tor^R_1(M,\lambda M),R)\cong \Ext^1_R(M,(\lambda M)^*)$; see for example \cite[page 7]{EG}; implies that $\Ext^1_R(M,(\lambda M)^*)=0$. But by \ref{fund}, $(\lambda M)^*\underset{st}\cong \Omega M$. Hence $\Ext^1_R(M,\Omega M)=0$ and so $M$ is free $R$-module.
\indent Conversely, if $M$ is generically free $R$--module, then $\Tor^{R_\fp}_1(M_\fp,(\lambda M)_\fp)=0$, for all $\fp\in \Ass R$. Hence $\Supp_R(\Tor^R_1(M,\lambda M))\cap\Ass_R(R)=\emptyset$ and we are done.
\end{proof}

\begin{defn}\label{CR}\emph{An $R$-module $M$ is said to have \emph{constant rank} $n$ on a subset $S$ of $\Spec R$ if $M_\fp$ is a free $R_\fp$-module with rank $n$ for all $\fp\in S$.}
\end{defn}

\begin{prop}\label{CR} Let $R$ be a generically Gorenstein unmixed local ring and let $M$ be a finitely generated $R$-module. Assume $M$ is horizontally linked to $\lambda M$ with $\Tor^R_1(M,\lambda M)=0$. If $\Ann_R(M)\neq 0$, then $\Ann_R(M)$ is geometrically linked to $(0:_R\Ann_R(M))$.
Moreover if $\lambda M$ is of constant rank on \emph{$\Ass_R(\lambda M)$}, then $\Ass_R(M)\cap \Ass_R(\lambda M)=\emptyset$.
\end{prop}
\begin{proof} Set $I=\Ann_R(M)$. By \ref{L1} and \ref{F1}, $\Ass_R(R/I)=\Ass_R(M)\subseteq \Ass(R)$. Hence, by \ref{MS4},
$R/I$ is a horizontally linked $R$-module and so $I$ is a linked ideal. Let $\fp \in \Ass_R(M)=\Ass_R(R/I)$, and set $J=(0:_RI)$.
Since by Lemma \ref{P1-1}, $M_\fp$ is a free $R_\fp$--module, $IR_p=0$, and since $J=(0:_RI)$, we have $JR_\fp=R_\fp$.
Hence $J\nsubseteq \fp$ and so $\fp \notin \Ass_R(R/J)$. This means that $I$ and $J$ have no common associated prime ideal.

Now assume that $\lambda M$ has constant rank on $\Ass_R(\lambda M)$. There is an exact sequence $0 \rightarrow M^* \rightarrow F \rightarrow \lambda M\rightarrow 0$, where $F$ is a free module. Choose $\fp, \fq \in \Ass(R)$ such that $\fp \in \Ass_RM$ and $\fq \in \Ass (R)\setminus \Ass_R(M)$. Note that
as $\Ann_R(M)$ is geometrically linked, Lemma \ref{L1} implies that such $\fq$ exists. Also note that by \ref{F1}, we have $\Ass_R(M) \subseteq \Ass(R)$ and $\Ass_R(\lambda M) \subseteq \Ass(R)$. Hence, localizing the exact sequence above at $\fp$ and $\fq$, we get $\rank(\lambda M)_{\fp}\neq \rank(\lambda M)_{\fq}$. Since $\lambda M$ has constant rank on $\Ass_R(\lambda M)$ we must have $(\lambda M)_{\fp} = 0$. Thus $\Ass_R(M)$ and $\Ass_R(\lambda M)$ have empty intersection.
\end{proof}

The assumption $\Ann_R(M)\neq 0$ in Proposition \ref{CR} is crucial. In section 4, we give an example of a module $M$ over a one-dimensional Gorenstein local ring $R$, such that $M$ is horizontally linked to $\lambda M$, $\Tor^R_1(M,\lambda M)=0$, and $\lambda M$ has constant rank on $\Ass_R(\lambda M)$, but $\Ass_R(M)\cap \Ass_R(\lambda M)\neq\emptyset$; see \ref{Ex2}.

Now we prove Theorem \ref{Int1}.

\begin{proof}[Proof of Theorem \ref{Int1}] Set $I=\Ann_R(M)$ and $J=\Ann_R(\lambda M)$.

(i)$\Rightarrow$(ii) follows from Proposition \ref{T1}.

(ii)$\Rightarrow$(i) and (ii)$\Rightarrow$(iv). We have $\Ass_R(R/I)\cap \Ass_R(R/J)=\emptyset$. Hence by Lemma \ref{L1}, $\Ass_R(M)\cap\Ass_R(\lambda M)=\emptyset$ and therefore, $M$ is free $R/I$-module, by Proposition \ref{T1}. Hence $\Ext^1_R(M,M)=0$, by Proposition \ref{P1}.

(iv)$\Rightarrow$(iii) follows from Proposition \ref{P1}.

(iii)$\Rightarrow$(ii) By Proposition \ref{CR}, $I$ is geometrically linked to $(0:_RI)$ and as $(0:_RI)=J$, we are done.
\end{proof}

As an application of Proposition \ref{T1} and Proposition \ref{CR}, we prove the following result, which is related to Huneke-Wiegand conjecture in a sense.

\begin{thm}\label{Int2} Let $R$ be a Cohen-Macaulay local domain of dimension $d\geq 1$ admitting a canonical module $\omega_R$.
Let $M$ be a torsion-free $R$-module such that $M\otimes_R\Hom_R(M,\omega_R)$ also is torsion-free $R$-module.
Assume $R\cong S/\fp$, where $S$ is a Gorenstein local ring and $\fp$ is a geometrically linked prime ideal of $S$.
If $\Ann_S(\Omega_SM)\neq 0$, then $M$ is a free $R$-module.
\end{thm}
Before the proof, we need the following lemma:
\begin{lem}\label{depth} Let $R$ be a generically Gorenstein unmixed local ring of dimension $d\geq 1$ and let $M$ be a horizontally linked $R$-module.
Assume $\Ann_RM$ is a non-zero prime ideal of $R$. Then $\Tor^R_1(M,\lambda M)=0$ if and only if $\Ann_RM$ is a geometrically linked ideal of $R$ and $\Ass_R(M\otimes_RM^*)\subseteq \Ass_R(M)$.
\end{lem}
\begin{proof} Let $F_1\rightarrow F_0 \rightarrow M \rightarrow 0$ be the minimal free presentation of $M$ and consider the exact sequence
$0\rightarrow M^* \rightarrow F_0^*\rightarrow \lambda M \rightarrow 0$. By applying $-\otimes_RM$ we get the exact sequence
\begin{equation}\tag{\ref{depth}.1}
0\rightarrow \Tor^R_1(M,\lambda M)\rightarrow M\otimes_RM^*\rightarrow M\otimes F_0^*.
\end{equation}
Now if $\Tor^R_1(M,\lambda M)=0$ then by Proposition \ref{CR}, $\Ann_RM$ is a geometrically linked ideal of $R$ and the exact sequence \ref{depth}.1
implies that $\Ass_R(M\otimes_RM^*)\subseteq \Ass_R(M)$.

Set $\fp = \Ann_RM$. Since $\fp$ is geometrically linked to $(0:_R\fp)$, one has $\fp \cap (0:_R\fp)=0$ and therefore $(0:_R\fp)\nsubseteq \fp$. Hence $\fp R_{\fp}=0$ and so $R_{\fp}$ is a field. Thus $M_{\fp}$ is a free $R_{\fp}$-module. Hence by \ref{P1-1} $\Supp_R(\Tor^R_1(M,\lambda M))\cap\Ass(R)=\emptyset$. By the exact sequence \ref{depth}.1 we have $\Ass_R(\Tor^R_1(M,\lambda M))\subseteq \Ass_R(M\otimes_RM^*)$ and since $\Ass_R(M\otimes_RM^*)\subseteq \Ass_R(M) =\{\fp\} \subset \Ass R$ so $\Ass_R(\Tor^R_1(M,\lambda M))=\emptyset$. Thus $\Tor^1_R(M,\lambda M)=0$.
\end{proof}

\begin{proof}[Proof of Theorem \ref{Int2}] Since $\fp \in \Ass(S)$, one has $\dim R = \dim S $. Therefore from the natural isomorphisms
\[\begin{array}{rl}\
M\otimes_R\Hom_R(M,\omega_R)&\cong M\otimes_S \Hom_R(M,\omega_R)\\
&\cong M\otimes_S \Hom_R(M,\Hom_S(R,S))\\
&\cong M\otimes_S\Hom_S(M\otimes_RR,S)\\
&\cong M\otimes_S\Hom_S(M,S),
\end{array}\]
we see that $\Ass_S(M\otimes_S\Hom_S(M,S))=\{\fp\}$. Since $M$ is a stable $S$-module, it is horizontally linked as $S$-module; see \ref{MS4}. Denote by $\lambda_SM$, the linkage of $M$ over $S$. Since $\Ass_S(M)=\Ass_S(M\otimes_S\Hom_S(M,S))=\{\fp\}$, we get $\Tor^S_1(M,\lambda_SM)=0$, by \ref{depth}.

Assume $\Ann_S(\Omega_SM)\neq 0$. Since by \ref{fund}, $\Omega_S M\cong \Hom_S(\lambda_SM,S)$, one has $\Ann_S(\lambda_SM)\neq 0$ and as $\Ass_S(M) = \{\fp\}$, $M$ has constant rank on $\Ass_S(M)$. Thus, it follows from Proposition \ref{CR} that $\Ass_S(M)\cap\Ass_S(\lambda_S M)=\emptyset$ and so $M$ is a free $S/\Ann_S(M)$-module, by Proposition \ref{T1}. But as $\Ann_S(M)=\fp$, $M$ is a free $R$--module.
\end{proof}

\begin{cor} Let $R$ be a Cohen-Macaulay local domain with canonical module $\omega_R$. Then $R$ is Gorenstein if and only if $R$ is a homomorphic image of a
reduced Gorenstein local ring $S$ with $\dim S = \dim R$ and $\Ann_S(\Omega_S\omega_R)\neq 0$.
\end{cor}
\begin{proof} If $R$ is Gorenstein then by setting $S=R$ we are done. Conversely, if $S$ is domain then we must have $S\cong R$ and so there is nothing to prove. Assume that $S$ is not domain and so there is a minimal prime ideal $\fp$ of $S$ such that $R\cong S/\fp$. Hence by \ref{MS4}, $\fp$ is a linked ideal of $S$.
Note that as $S$ is reduced, then $\fp$ is a geometrically linked ideal. Indeed, by setting $I=(0:_S\fp)$ we get $I\cap\fp=0$ because otherwise $S$
has a nilpotent element which is a contradiction. Since $\omega_R\otimes_R\Hom_R(\omega_R,\omega_R)$ is torsion-free, Theorem \ref{Int2} implies that $\omega_R$ is a free $R$-module and so $R$ is Gorenstein.
\end{proof}

\begin{prop}\label{P2} Let $R$ be a generically Gorenstein unmixed local ring, and let $M$ be a horizontally linked $R$--module.
Assume \emph{$\Ann_R(M) \neq 0$} and that \emph{$\Tor^R_1(M,\lambda M)=0$}. If there exists a positive odd integer $n$ such that $\Omega^n M$ has constant rank on \emph{$\Ass_R(\Omega^n M)$}, and the sequence of the Betti numbers $\beta_0(M),\cdots,\beta_{n-1}(M)$ is non-decreasing, then $\Ass_R(M)\cap\Ass_R(\lambda M)=\emptyset$.
\end{prop}
\begin{proof} Set $\beta_i=\beta_i(M)$.
Let $\cdots \rightarrow F_1 \rightarrow F_0 \rightarrow M \rightarrow 0$ be the minimal free resolution of $M$ and consider the exact sequence
$$0\rightarrow \Omega^nM \rightarrow F_{n-1} \rightarrow \cdots \rightarrow F_0 \rightarrow M \rightarrow 0.$$
Note that by \ref{CR}, $\Ass(R) \setminus \Ass_R(M)\neq \emptyset$. Choose $\fp,\ \fq \in \Ass(R)$ such that $\fp \in \Ass_R(M)$ and $\fq \in \Ass(R) \setminus \Ass_R(M)$. Localizing the last exact sequence at $\fp$ and $\fq$, by \ref{P1-1} we get the equations
\begin{equation}\tag{\ref{P2}.1}
\rank_{R_\fp} (\Omega^nM)_{\fp} = \sum^{n-1}_{i=0}(-1)^i \beta_i-\rank_{R_\fp}M_\fp
\end{equation}
and
\begin{equation}\tag{\ref{P2}.2}
\rank_{R_\fq} (\Omega^nM)_{\fq} = \sum^{n-1}_{i=0}(-1)^i \beta_i.
\end{equation}
Therefore, as $\Omega^n M$ has constant rank on \emph{$\Ass_R(\Omega^n M)$}, we must have either $\Ass_R(\Omega^nM) \subseteq \Ass_R(M)$ or $\Ass_R(M)\cap \Ass_R(\Omega^nM) = \emptyset$. If  $\Ass_R(\Omega^nM) \subseteq \Ass_R(M)$ then $(\Omega^nM)_{\fq}=0$ and \ref{P2}.2 implies that
$ \sum^{n-2}_{i=0}(-1)^{i-1}\beta_i=\beta_{n-1}\geq\beta_{n-2}$. Canceling $\beta_{n-2}$ and continuing in this way,
we eventually obtain $\beta_0\leq 0$ which is impossible.

Hence $\Ass_R(M)\cap \Ass_R(\Omega^nM) = \emptyset$ and by localizing the exact sequence
$$0\rightarrow \Omega^nM \rightarrow F_{n-1} \rightarrow \cdots \rightarrow F_1 \rightarrow \Omega M \rightarrow 0$$
at any $\fp \in \Ass_R(M)$ and noting that $\beta_i$ is non-decreasing, we get $\rank_{R_\fp} (\Omega M)_\fp = \sum^{n-1}_{i=1}(-1)^{i-1}\beta_i\leq 0$.
It follows that $(\Omega M)_\fp =0$, for all $\fp \in \Ass_R(M)$. As $\Omega M \cong (\lambda M)^*$, by \ref{fund}, so $\Ass_R(M) \cap \Ass_R(\lambda M) = \emptyset$.
\end{proof}

\section{One-dimensional gorenstein rings}

In this section we give an example of a horizontally linked module $M$ over a one-dimensional Gorenstein local ring $R$ such that $\Tor^R_1(M,\lambda M)=0$,
$\lambda M$ has constant rank on $\Ass_R(\lambda M)$, but $\Ass_R(M)\cap\Ass_R(\lambda M)\neq\emptyset$. First we prove a duality between $\Ext^*_R(M,M)$ and $\Tor^R_*(M,\lambda M)$; see \ref{p3}, which has its own interest.

\begin{lem}\label{PL1} Let $(R,\fm)$ be a Cohen-Macaulay local ring of dimension $d\geq 1$, admitting a canonical module $\omega_R$. Let $M$ and $N$ be maximal
Cohen-Macaulay $R$--modules such that $\Tor^R_i(M,N)$ is of finite length, for all $i\geq 1$. Set \emph{$(-)^\vee = \Hom_R(-,E(R/\fm))$} and
\emph{$(-)^\dag = \Hom_R(-,\omega_R)$}. Then \emph{$\Ext^{i+d}_R(M,N^\dag)\cong \Tor^R_i(M,N)^\vee$}, for all $i\geq 1$.
\end{lem}
\begin{proof} There is a third quadrant spectral sequence $$\E^{p,q}_2=\Ext^p_R(\Tor^R_q(M,N),\omega_R)\underset{p}\Longrightarrow \Ext^{p+q}_R(M,N^\dag).$$
As $\Tor^R_q(M,N)$ is of finite length, for all $q\geq 1$, it follows that $\Ext^p_R(\Tor^R_q(M,N),\omega_R)=0$, for all $p\neq d$ and all $q\geq 1$. Therefore
$$\E^{p,q}_2= \left\lbrace
           \begin{array}{c l}
              \Ext^p_R(M\otimes_RN,\omega_R)\ \ & \text{ \ \ $0\leq p\leq d$,}\\
              \Ext^d_R(\Tor^R_q(M,N),\omega_R)\ \   & \text{   \ \ $q>0$,}\\
              0\ \   & \text{  \ \ $\textrm{otherwise}$.}
           \end{array}
        \right.$$\\
Therefore, as the map $d^{p,q}_r:\E^{p,q}_r\rightarrow \E^{p+2,q-1}_r$ is of bidegree $(r,1-r)$, for all $r\geq 2$, one can easily verify that that $d^{p,q}_r =0$, for all $p,q$, and $r$. Hence $\E^{p,q}_2=\E^{p,q}_{\infty}$, for all $p$ and $q$ and so $\Ext^{q+d}_R(M,N^\dag)\cong \Ext^d_R(\Tor^R_q(M,N),\omega_R)$ for all $q>0$. Now, the Local Duality Theorem (see \cite[Theorem 3.5.8]{BH}) implies that $\Ext^{d}_R(\Tor^R_q(M,N),\omega_R)\cong \Gamma_{\fm}(\Tor^R_q(M,N))^\vee =\Tor^R_q(M,N)^\vee$.
\end{proof}

\begin{prop}\label{p3} Let $(R,\fm)$ be a Gorenstein local ring of dimension one, and let $M$ be a maximal Cohen-Macaulay $R$--module.
Assume either $\Tor^R_i(M,\lambda M)$ is of finite length, for all $i\geq 1$, or \emph{$\Tor^R_1(M,\lambda M)=0$}.
Then for all $i>0$, \emph{$$\Ext^i_R(M,M)\cong \Hom_R(\Tor^R_i(M,\lambda M),E(R/\fm)).$$}
\end{prop}
\begin{proof} Note that if $\Tor^R_1(M,\lambda M)=0$ then $M$ is locally free on the punctured spectrum of $R$, by \ref{P1-1}, and so $\Tor^R_i(M,\lambda M)$ is of finite length, for all $i\geq 1$. As $R$ is Gorenstein, then we have $(\lambda M)^\dag \cong(\lambda M)^*$ and $(\lambda M)^*$ is stably isomorphic to $\Omega M$. Thus we get $\Ext^{i+1}_R(M,(\lambda M)^*)\cong \Ext^i_R(M,M)$, for all $i>0$. Now the result follows by Lemma \ref{PL1}.
\end{proof}

\begin{cor}\label{Cor2} Let $R$ be a Gorenstein local ring of dimension one and $M$ be a horizontally linked $R$-module with $\Tor^R_1(M,\lambda M)=0$.
Then for all $n\geq 0$ and all $i\geq 1$, $\Tor^R_i(M,\lambda M)\cong\Tor^R_i(\Omega^nM,\lambda\Omega^nM)$. In particular, $\Tor^R_1(M,\lambda M)=0$ if and only if $\Tor^R_1(\Omega^nM,\lambda\Omega^nM)=0$, for some (and hence all) $n\geq 0$.
\end{cor}
\begin{proof} For each $n\geq 0$ and $i \geq 1$, one has $\Ext^i_R(M,M)\cong\Ext^i_R(\Omega^nM,\Omega^nM)$. Thus the result follows from Proposition \ref{p3}.
\end{proof}

In the following example, we use Macaulay 2 for the computations.

\begin{exam}\label{Ex2} \emph{Let $k$ be a field and $R=k[[x,y,z,t]]/(xy,yz,zt,xt+yt+t^2,x^2+xz+xt)$, which is a Gorenstein local ring of dimension one.
Let $I=(x,z)$ and $J=(y)$. Then $R/I$ is horizontally linked to $R/J$. One has $\Ass(R)=\{(z,y,x+t),(z,x,y+t),(x,z,t),(t,y,x),(t,y,x+z)\}$, $\Ass_R(R/I)=\{(z,x,y+t),(x,z,t)\}$ and $\Ass_R(R/J)=\{(z,y,x+t),(t,y,x),(t,y,x+z)\}$. Hence $I$ is geometrically linked to $J$ and so
$\Tor^R_1(R/I,R/J)=0$. Set $M = \lambda \Omega (R/I) = \lambda I$. Then $M$ is horizontally linked and $\lambda M \cong I$.
Hence $\Ann_R(\lambda M)=J$ and note that $\lambda M$ is not free $R/J$-module.  By the last corollary we have $\Tor^R_1(M,\lambda M)=0$.
Also note that $\lambda M$ has constant rank equal to one on $\Ass_R(\lambda M)=\Ass_R(R/J)$. Therefore we must have $\Ann_R(M)=0$ because
otherwise $\lambda M$ would be a free $R/J$-module, by \ref{CR}. Therefore $\Ass_R(M)\cap\Ass_R(\lambda M)\neq \emptyset$.}
\end{exam}

Recall that a \emph{complete resolution} of a finitely generated module $M$ over a local ring $(R,\fm)$ is a diagram $\emph{\textbf{T}}\overset{\nu}\longrightarrow \emph{\textbf{F}}\overset{\pi}\longrightarrow M$, where $\emph{\textbf{F}}$ is a free resolution of $M$, and
$\emph{\textbf{T}}$ is an exact complex of finitely generated free modules $$\emph{\textbf{T}}:\cdots\longrightarrow T_{i+1}\overset{d_{i+1}}\longrightarrow T_i\overset{d_i}\longrightarrow T_{i-1}\longrightarrow\cdots,$$ such that $\Hom_R(\emph{\textbf{T}},R)$ is exact, $\nu$ is a morphism of complexes, and $\nu_i$ is bijection for all $i\gg0$. A complete resolution $\emph{\textbf{T}}$ of $M$ is called \emph{minimal} if $\Im(d_i)\subseteq \fm T_{i-1}$, for all $i \in \mathbb{Z}$. It is well known that if $R$ is Gorenstein, then every finitely generated $R$-module admits a minimal complete resolution. In particular, if $M$ is maximal Cohen-Macaulay, then $M\cong \Im(d_0)$.

\begin{thm}Let $R$ be a Gorenstein local domain of dimension one and let $M$ be a torsion-free $R$-module such that $M\otimes_RM^*$ also is a torsion-free $R$-module. Assume $R\cong S/\fp$, where $S$ is a Gorenstein local ring and $\fp$ is a geometrically linked prime ideal of $S$. Let ${\textbf{T}}$ be a minimal complete resolution of $M$ over $S$. If $\Ann_S(\Omega^{-1}_SM)\neq 0$, then $M$ is a free $R$-module.
\end{thm}
\begin{proof} As we saw in the proof of \ref{Int2}, $M$ is horizontally linked over $S$ and one has $\Tor^S_1(M,\lambda_S M)=0$. Set $N=\Omega^{-1}_SM$. Since $\Ann_S(N)\neq 0$, $N$ is an stable $S$-module and so it is horizontally linked $S$-module; see \ref{MS2}. Note that as $M\cong \Omega_SN$, one has $\Tor^S_1(N,\lambda_S N)=0$, by \ref{Cor2}. Also, note that since $\Ass_S(M)=\{\fp\}$, $M$ has constant rank on $\Ass_S(M)$. Therefore, by \ref{P2} and \ref{T1}, $N$ is free over $S/\Ann_S(N)$ and $\lambda_S N$ is free over $S/\Ann_S(\lambda_S N)$. As $M\cong \Hom_S(\lambda_SN,S)$, it follows that $\Ann_S(\lambda_S N)=\fp$, and $M$ is isomorphic to a direct sum of copies of $\Hom_S(S/\fp,S)$. But, since $R$ and $S$ are Gorenstein, $\Hom_S(S/\fp,S)\cong  R$. Hence $M$ is a free $R$-module.
\end{proof}

\section{Modules horizontally linked to the syzygies of the residue field}

As mentioned before, a module is horizontally linked if and only if it is stable and a syzygy module. Therefore for an $R$-modules $M$ and each $n\geq 1$, $\Omega^nM$ is horizontally linked if and only if $\Omega^nM$ is stable. In this section we are interested in the properties of modules that are horizontally linked to the syzygies of the residue field. We start with the following lemma that we need for the rest of the paper.

\begin{lem}\label{stbl} Let $R$ be a local ring and let $M$ be a non-zero finitely generated $R$-module.
\begin{enumerate}[\rm(i)]
\item If $M = F\oplus N$ with $F$ a free $R$-module, then $\lambda M \cong \lambda N$.
\item If $M$ is non-free and a syzygy module, then $\lambda M$ is stable.
\end{enumerate}
\end{lem}
\begin{proof} We only prove (ii). Assume $M=F\oplus N$ where $F$ is a free $R$-module and $N$ is a non-zero stable $R$-module. By (i), $\lambda M = \lambda N$ and by \ref{MS2}, $N$ is horizontally linked. Now \ref{MS3} implies that $\lambda M$ is stable.
\end{proof}

\begin{lem}\label{syz} Let $R$ be a local ring and let $M$ be a non-zero $R$-module. If $n>\max\{0, \depth(R)-\depth(M)\}$,
then $\Omega^n M$ is stable $R$-module; see for example \cite[Corollary 1.2.5]{Av}.
\end{lem}

\begin{thm} \label{th1} Let $R$ be a local ring which is not a field. Let $M$ be a non-zero $R$-module and let $n\geq 0$. Assume $\lambda \Omega^nM \cong X \oplus Y$ for some non-zero $R$-module  $X$. Then:
\begin{enumerate}[\rm(i)]
\item $\Tor^R_n(M,X)\neq 0$.
\item If $\pd(X)< \infty$ then $\depth R\geq 1$ and $n\leq \pd(X)< \depth(R)-\depth(M)$.
\end{enumerate}
\end{thm}
\begin{proof} (i) If $n=0$ we have nothing to prove. Suppose $n\geq 1$ and assume to the contrary that $\Tor_n^R(M,X)=0$.
Consider the short exact sequence
\begin{equation} \tag{\ref{th1}.1}
0 \to \lambda\Omega^nM \to F \to \Tr\Omega^nM \to 0
\end{equation}
which follows from the definition of the linkage, where $F$ is free.
By using \cite[Lemma 3.1]{T}, we obtain the following short exact sequences from (\ref{th1}.1)
\begin{equation}\tag{\ref{th1}.2}
0 \to X \to B \to \Tr\Omega^n M \to 0
\end{equation}
\begin{equation}\tag{\ref{th1}.3}
0 \to F \to A\oplus B \to \Tr\Omega^n M \to 0,
\end{equation}
where $A$ and $B$ are $R$-modules. Since $M$ and $X$ are non-zero, and that $n\geq 1$, $\Ext^1_R(\Tr\Omega^n M,R)=0$, by \ref{ExtTr}.
This implies that (\ref{th1}.3) splits, so $A\oplus B \cong F \oplus \Tr\Omega^n M$. It follows from the Auslander- Bridger four term exact sequence \cite[Theorem 2.8]{AB} that $\Ext^1_R(\Tr \Omega^n M, X) \hookrightarrow \Tor_n^R(M,X)=0$. Thus (\ref{th1}.2) splits too, and so $B\cong X \oplus \Tr\Omega^n M$. Hence
\begin{equation}\tag{\ref{th1}.4}
F  \oplus \Tr\Omega^n M \cong A\oplus X \oplus \Tr\Omega^n M
\end{equation}
Now the cancelation property (see \cite[Corollary 1.16]{LW}) applied to (\ref{th1}.4) shows that $F\cong A \oplus X$, i.e., $X$ is free; which is a contradiction, by \ref{stbl} (ii). This proves (i).

(ii) Assume $\pd(X)< \infty$. If $\depth R=0$ then Auslander-Buchsbaum formula implies that $X$ is free and by the last part $n=0$.
By Lemma \ref{stbl} (i), we may assume that $M$ is stable.
Let $F_1\rightarrow F_0 \rightarrow M \rightarrow 0$ be the minimal presentation of $M$. Then $F_0^*\rightarrow F_1^* \rightarrow \Tr M \rightarrow 0$ is the minimal presentation of $\Tr M$. But since $X \oplus Y = \lambda M = \Omega\Tr M$ and $X$ is free, this contradicts \ref{syz}. Thus $\depth R \geq 1$.

For the next part, we proceed by induction on $\depth M$. If $\depth_R M =0$ then we need to show that $\pd(X)<\depth R$. But, as $\depth R \geq 1$ and $X$ is a
syzygy module, $\depth_R X \geq 1$. Hence by the Auslander-Buchsbaum formula $\pd(X) < \depth R$.

Let $\depth M \geq 1$ and choose $x$ to be regular on $R$, $M$ and $\lambda\Omega^n M$. As $\Tor^R_i(R/(x),M)=\Tor^R_i(R/(x),\lambda\Omega^n M)=0$ for all $i>0$, it follows that $\lambda_{R/(x)}\Omega^n(M/xM) \cong \lambda \Omega^nM/x\lambda\Omega^n M \cong X/xX\oplus Y/xY$. By induction, $\pd_{R/(x)}(X/xX)<\depth R/(x) - \depth_{R/(x)} M/xM= \depth R-\depth_RM$. Since $\pd_{R/(x)}(X/xX)=\pd(X)$, we are done.
\end{proof}

Now Theorem \ref{Int4} follows from the next corollary.

\begin{cor} \label{Cor3} Let $(R,\fm,k)$ be a local ring, $L$ a finite length $R$-module, and $n\geq 0$ an integer. Assume that
$\lambda \Omega^nL \cong X \oplus Y$ for some non-zero $R$-module $X$. Then the following statements are equivalent.
\begin{enumerate}[\rm(i)]
\item $\pd(X)<\infty$.
\item $\pd(X)=n$.
\item $\depth R > n$.
\end{enumerate}
\end{cor}
\begin{proof} (i)$\Longrightarrow$(iii) Follows from \ref{th1} (ii).

(iii)$\Longrightarrow$(ii) Assume $\depth R>n$. Note that as $\Tor^R_n(X,L)\neq 0$, by \ref{th1} (i), therefore $\pd(X)\geq n$. We show that $\pd(\lambda\Omega^n L)\leq n$ and hence $\pd(X)\leq n$. Consider the exact sequence which is part of a minimal free resolution of $L$: $$\textbf{G}= (0\rightarrow \Omega^nL \rightarrow F_{n-1} \rightarrow \cdots \rightarrow F_1 \rightarrow F_0 \rightarrow L \rightarrow 0)$$
Since $\depth R>n$ and $L$ is finite length, by dualizing $\textbf{G}$, we obtain the exact sequence:
$$\textbf{G}^{\ast}=(0\rightarrow F_0^* \rightarrow F_1^* \rightarrow \cdots \rightarrow F_{n-1}^* \rightarrow (\Omega^nL)^* \rightarrow 0).$$
It follows that $\pd(\Omega^n L)^* \leq n-1$ and so $\pd(\lambda \Omega^n L)\leq n$.

(ii)$\Longrightarrow$(i) It is clear.
\end{proof}

\begin{cor} \label{5.4} Let $R$ be a local ring and let $M$ be a non-zero $R$-module of finite injective dimension. Let $n\geq 0$ be a positive integer such that $\Omega^nM \cong X \oplus Y$ for some non-zero $R$-module $X$ with $\gd(X)<\infty$. Then $\pd(X)<\infty$.
\end{cor}
\begin{proof} Choose $i\geq \max\{n,\depth R-\depth M \}$. Then $\Omega^iM \cong \Omega^{i-n}X\oplus \Omega^{i-n}Y$.
Hence $\lambda\Omega^iM \cong \lambda\Omega^{i-n}X\oplus \lambda\Omega^{i-n}Y$ and by \ref{th1} (i), if $\lambda\Omega^{i-n}X \neq 0$ then $\Tor^R_i(M,\lambda\Omega^{i-n}X)\neq 0$. But as $\id(M)<\infty$ and $\Omega^{i-n}X$ is totally reflexive, $\Tor^R_i(M,\lambda\Omega^{i-n}X)= 0$; see \cite[Corollary 2.4.4]{C}. Thus we must have $\lambda\Omega^{i-n}X=0$ and therefore $\Omega^{i-n}X$ is zero or free.
\end{proof}

\begin{cor} (\cite[Theorem 3.2 (i)]{Holm}) Let $R$ be a local ring and let $M$ be a non-zero $R$-module.
Assume $\id(M)<\infty$ and $\gd(M)<\infty$. Then $R$ is Gorenstein.
\end{cor}
\begin{proof} We conclude from \ref{5.4} that $\pd(M)<\infty$. Now Foxby's result \cite{F}, implies that $R$ is Gorenstein.
\end{proof}

Our aim in the rest of the paper is to prove a counterpart of \ref{Cor3} in case of finite Gorenstein and complete intersection dimension.
Recall from \cite{AGP} that a diagram of local ring maps $R \to R' \twoheadleftarrow Q$ is called a {\em quasi-deformation} provided that $R \to R'$ is flat and the kernel of the surjection $R' \twoheadleftarrow Q$ is generated by a $Q$-regular sequence. The \emph{complete intersection dimension} of $M$ is:
\begin{equation*}
\CI(M) = \inf\{ \pd_Q(M \otimes_{R} R') - \pd_Q(R') \ | \
R \to R' \twoheadleftarrow Q \ {\text{is a quasi-deformation}}\}.
\end{equation*}

\begin{lem} \label{Gdimlem} Let $R$ be a local ring, $M$ an $R$-module, and let $n\geq \depth(R)$. Then the following conditions are equivalent:
\begin{enumerate}[\rm(i)]
\item $\gd(\Tr \Omega^n M)<\infty$.
\item $\gd(M)<\infty$.
\item $\gd(\Tr \Omega^n M)\le 0$.
\item $\gd(\Omega^n M)\le 0$.
\end{enumerate}
\end{lem}
\begin{proof} Note that (ii) $\Longleftrightarrow$ (iv) and (iii) $\Longrightarrow$ (i) follow by definition. Moreover as the transpose of a totally reflexive module is also totally reflexive, (iii) $\Longleftrightarrow$ (iv) is clear. Hence it is enough to show that (i) $\Longrightarrow$ (iii).

Set $t=\depth R$ and assume $\gd(\Tr \Omega^n M)<\infty$. Note that $\depth(\Omega^nM)_p\geq \min\{n,\depth (R_p)\}$, for all $\fp \in \Supp_R(\Omega^nM)$.  It follows from \cite[Theorem 5.8 (1)]{CGSZ} that $\Ext^i_R(\Tr \Omega^n M, R)=0$ for all $i=1, \ldots, t$. Since $\gd(\Tr \Omega^n M)=\sup\{i \in \ZZ: \Ext^i_R(\Tr \Omega^n M,R)\neq 0\}$; see \cite[Theorem 1.2.7 (iii)]{C}, if $\gd(\Tr \Omega^n M)> 0$, then we must have $\gd(\Tr \Omega^n M)> t$, which is impossible, by Auslander-Bridger formula; see \cite{C}. Consequently $\gd(\Tr \Omega^n M)\leq 0$ and (iii) follows.
\end{proof}

\begin{thm}\label{T4} Let $(R,\fm,k)$ be a local ring, let $M$ be an $R$-module of finite length, and let $n\geq 0$ be an integer.
Let $\hd$ stand for one of projective, complete intersection or Gorenstein dimension, and assume $\hd_R(M)=\infty$.
Then the following conditions are equivalent:
\begin{enumerate}[\rm(i)]
\item $\depth(R)\leq n$.
\item $\hd(\lambda\Omega^nM)=\infty$.
\end{enumerate}
In particular $\depth R=\inf\{n| \hd(\lambda\Omega^n M)=\infty\}$.
\end{thm}
\begin{proof} (i)$\Longrightarrow$(ii) Let $n\geq\depth R$ and assume to the contrary that $\hd(\lambda\Omega^n M)<\infty$. If $\hd(\lambda\Omega^n M) = \gd(\lambda\Omega^n M)$, then it follows that $\gd(\Tr\Omega^n M)<\infty$ and by Lemma \ref{Gdimlem}, $\gd(\Omega^n M)<\infty$ and hence $\gd(M)<\infty$ which is a contradiction.

If $\hd(\lambda\Omega^n M)=\ci_R(\lambda\Omega^n M)$, then $\gd(\lambda\Omega^n M)<\infty$ and so $\gd(\Tr\Omega^n M)<\infty$. Hence by Lemma \ref{Gdimlem}, $\gd_R(\Tr\Omega^nM)\leq 0$ and so $\ci_R(\Tr\Omega^nM) \leq 0$.
Thus $\ci_R(\Omega^n M)=0$, by \cite[Lemma 3.2 (i)]{CST}. Hence $\ci_R(M)<\infty$, which is a contradiction.

(ii)$\Longrightarrow$(i) First assume $n=0$ and $\hd(\lambda M)=\infty$. We show that $\depth(R)=0$. Consider the exact sequence $0\rightarrow M^*\rightarrow F\rightarrow \lambda M\rightarrow 0$, where $F$ is a free $R$-module. Since $\hd(\lambda M)=\infty$, one has $M^*\neq 0$. Therefore, as $M$ is a finite length $R$-module, one gets $\depth(R)=0$.

Now assume $n>0$. It is enough to show that if $\depth(R)>n$ then $\hd(\lambda\Omega^nM)<\infty$ and this follows by the same argument of \ref{Cor3} (iii)$\Longrightarrow$(ii).
\end{proof}

\begin{cor}
Let $(R,\fm,k)$ be a local ring, $n\geq\depth R$ be a positive integer. Then the following statements are true.
\begin{enumerate}[\rm(i)]
\item{$R$ is Gorenstein if and only if $\gd_R(\lambda\Omega^nk)<\infty$.}
\item{$R$ is a complete intersection if and only if $\ci_R(\lambda\Omega^nk)<\infty$.}
\end{enumerate}
\end{cor}

\begin{acknowledgement}
\emph{Part of this work was completed when Celikbas and Takahashi visited
Centre de Recerca Matem\`{a}tica during the conference; Opening Perspectives
in Algebra, Representations, and Topology, in May 2015. The authors acknowledge
the kind hospitality of CRM and thank the institute for providing an excellent research
environment.}

\emph{Part of this work was completed when Gheibi visited Celikbas at the University of Connecticut
in October 2015. Gheibi thanks Jerzy Weyman and Roger Wiegand for supporting his visit, and is grateful for the kind hospitality of the UConn Department of Mathematics.}

\emph{The authors thank Roger Wiegand for his careful comments and suggestions on the manuscript, and Thanh Vu for his help with Example \ref{Ex2}.}
\end{acknowledgement}

\bibliographystyle{99}

\end{document}